\RequirePackage{fix-cm}
\documentclass[11pt, a4paper, oneside, DIV11, final, pagebackref]{amsart}
\usepackage{fixltx2e}
\usepackage[notref,notcite]{showkeys}
\usepackage{dsfont}
\usepackage[latin1]{inputenc}
\usepackage[T1]{fontenc}
\usepackage{amssymb}
\usepackage[stretch=10]{microtype}
\usepackage{mathtools}
\usepackage[DIV11, headinclude=true]{typearea}
\usepackage[hidelinks, linktoc=all]{hyperref}

\providecommand{\href}[2]{#2}

\providecommand*{\backref}{}
\providecommand*{\backrefalt}{}
\renewcommand*{\backref}[1]{}
\renewcommand*{\backrefalt}[4]{%
	\ifcase #1 %
	\or
	  Cited page~#2.
	\else
	  Cited pages~#2.
	\fi
}

\newcommand\MTkillspecial[1]{
  \bgroup
  \catcode`\&=9
  \let\\\relax%
  \scantokens{#1}%
  \egroup
}
\newcommand\DeclarePairedDelimiterMultiline[3]{
  \DeclarePairedDelimiter{#1}{#2}{#3}
  \reDeclarePairedDelimiterInnerWrapper{#1}{star}{
    \mathopen{##1\vphantom{\MTkillspecial{##2}}\kern-\nulldelimiterspace\right.}
    ##2
    \mathclose{\left.\kern-\nulldelimiterspace\vphantom{\MTkillspecial{##2}}##3}}
}

\newcommand{\breakdots}{,\allowbreak\dotsc,\allowbreak}
\newcommand{\given}{\mid}
\newcommand{\bbun}{\mathds{1}}
\newcommand{\boF}{\mathcal{F}}

\newcommand{\boS}{\mathcal{S}}
\newcommand{\E}{\mathbb{E}}
\newcommand{\Pbb}{\mathbb{P}}

\newcommand{\N}{\mathbb{N}}

\newcommand{\dd}{\mathop{}\!\mathrm{d}}
\DeclarePairedDelimiter{\paren}{(}{)}
\DeclarePairedDelimiterMultiline{\abs}{\lvert}{\rvert}
\DeclarePairedDelimiterMultiline{\norm}{\lVert}{\rVert}
\newcommand{\st}{\::\:}

\renewcommand{\epsilon}{\varepsilon}

\renewcommand{\phi}{\varphi}
\renewcommand{\leq}{\leqslant}
\renewcommand{\geq}{\geqslant}

\newtheorem{thm}{Theorem}

\newtheorem{definition}[thm]{Definition}
\newtheorem{lem}[thm]{Lemma}

\newtheorem*{prop*}{Proposition}

\theoremstyle{definition}

\newtheorem{rmk}[thm]{Remark}

\title[Concentration inequalities for Markov chains]{Subgaussian concentration inequalities for geometrically ergodic Markov chains}
\author{J\'er\^ome Dedecker and S\'ebastien Gou\"ezel}

\address{Laboratoire MAP5 UMR CNRS 8145, Universit\'e Paris Descartes, Sorbonne Paris Cit\'e}
\email{jerome.dedecker@parisdescartes.fr}

\address{IRMAR, CNRS UMR 6625,
Universit\'e de Rennes 1, 35042 Rennes, France}
\email{sebastien.gouezel@univ-rennes1.fr}

\date{}

\begin{document}

\begin{abstract}
We prove that an irreducible aperiodic Markov chain is geometrically
ergodic if and only if any separately bounded functional of the stationary
chain satisfies an appropriate subgaussian deviation inequality from its
mean.
\end{abstract}

\maketitle

\bigskip

Let $K(x_0,\dotsc, x_{n-1})$ be a function of $n$ variables, which is
separately bounded in the following sense: there exist constants $L_i$ such
that for all $x_0,\dotsc, x_{n-1}, x'_i$,
\begin{equation}
\label{eq:separ_bounded}
  \abs{K(x_0,\dotsc,x_{i-1},x_i,x_{i+1},\dotsc, x_{n-1})- K(x_0,\dotsc,x_{i-1},x'_i,x_{i+1},\dotsc, x_{n-1})}
  \leq L_i.
\end{equation}
It is well known that, if the random variables $X_0,X_1,\dotsc$ are i.i.d.,
then $K(X_0,\dotsc, X_{n-1})$ satisfies a subgaussian concentration
inequality around its average, of the form
\begin{equation}
\label{eq_concentration}
  \Pbb( \abs{K(X_0,\dotsc, X_{n-1}) - \E K(X_0,\dotsc, X_{n-1})}>t) \leq 2 e^{-2 t^2 / \sum L_i^2},
\end{equation}
see for instance~\cite{McDiarmid}.

Such concentration inequalities have also attracted a lot of interest for
dependent random variables, due to the wealth of possible applications. For
instance, Markov chains with good mixing properties have been considered, as
well as weakly dependent sequences.

A particular instance of function $K$ is a sum $\sum f(x_i)$ (also referred
to as an additive functional). In this case, one can hope for better
estimates than~\eqref{eq_concentration}, involving for instance the
asymptotic variance instead of only $L_i$ (Bernstein-like inequalities).
However, for the case of a general functional $K$, estimates of the
form~\eqref{eq_concentration} are rather natural.

Under very strong assumptions ensuring that the dependence is uniformly small
(say, uniformly ergodic Markov chains, or $\Phi$-mixing dependent sequences),
subgaussian concentration inequalities are well known
(see~\cite{rio_concentration} for the extension of~\eqref{eq_concentration}
and~\cite{Samson} for other concentration inequalities). For additive
functionals, Lezaud~\cite[p~861]{lezaud} proved a Prohohorov-type inequality
under a spectral gap condition in ${\mathbb L}^2$, from which  a subgaussian
bound follows. However, there are very few such results under weaker
assumptions (say, geometrically ergodic Markov chains, or $\alpha$-mixing
dependent sequences), where other type of exponential bounds are more usual
(let us cite~\cite{merlevede_peligrad_rio} for $\alpha$-mixing sequences
and~\cite{adamczak_bednorz} for geometrically ergodic Markov chains; see also
the references in those two papers for a quite complete picture of the
literature). As an exception, let us mention the result of Adamczak, who
proves in~\cite{adamczak_markov} subgaussian concentration inequalities for
geometrically ergodic Markov chains under the additional assumptions that the
chain is strongly aperiodic and that the functional $K$ is invariant under
permutation of its variables.

Our goal in this note is to prove subgaussian concentration inequalities for
aperiodic geometrically ergodic Markov chains, extending the above result
of~\cite{adamczak_markov}. Such a setting has a wide range of applications,
in particular to MCMC (see for instance Section~3.2
in~\cite{adamczak_bednorz}). Our proof is mainly a reformulation in
probabilistic terms of the proof given
in~\cite{gouezel_chazottes_concentration} for dynamical systems. It is based
on a classical coupling estimate (Lemma~\ref{lem:coupling} below), but used
in an unusual way along an unusual filtration (the relationship between
coupling and concentration has already been explored
in~\cite{chazottes_redig}). Similar results can also be proved for Markov
chains that mix more slowly (for instance, if the return times to a small set
have a polynomial tail, then polynomial concentration inequalities hold). The
interested reader is referred to the
articles~\cite{gouezel_chazottes_concentration} and~\cite{gouezel_melbourne}
where such results are proved for dynamical systems: the proofs given there
can be readily adapted to Markov chains using the techniques we describe in
the current paper (the only difficulty is to prove an appropriate coupling
lemma extending Lemma~\ref{lem:coupling}). Since the main case of interest
for applications is geometrically ergodic Markov chains, and since the proof
is more transparent in this case, we only give details for this situation.

\medskip

Our results are valid for Markov chains on a general state space $\boS$, but
they are already new and interesting for countable state Markov chains. The
reader who is unfamiliar with general state space Markov chains is therefore
invited to pretend that $\boS$ is countable. We chose to present our results
for general state space firstly because of the wealth of applications, and
secondly because of a peculiarity of general state space that does not exist
for countable state space: there is a distinction between strongly aperiodic
and aperiodic chains, and several mixing results only apply in the strongly
aperiodic case (i.e., $m=1$ in Definition~\ref{defn:geom_erg} below) while
our argument always applies.

\bigskip

From this point on, we consider an irreducible aperiodic positive Markov
chain $(X_n)_{n \geq 0}$ on a general state space $\boS$, which we assume as
usual to be countably generated. We refer to the books~\cite{nummelin2}
or~\cite{meyn_tweedie} for the classical background on Markov chains on
general state spaces. Let us nevertheless recall the meaning of some of the
above terms, since it may vary slightly between sources.

First, we are given a measurable transition kernel $P$ of the chain, that is,
for any measurable set $A$ in $\boS$,
 \[
P(x, A)={\mathbb E}\left (\bbun_{X_1 \in A} \given X_0=x\right ) .
 \]
Starting from any point $x_0$, we obtain a chain $X_0=x_0, X_1, X_2,\dotsc$,
where $X_i$ is distributed according to the measure $P(X_{i-1}, \cdot)$. This
chain is irreducible, aperiodic and positive if there exists a (necessarily
unique) stationary probability measure $\pi$ such that, for all $x$, all set
$A$ with $\pi(A)>0$ and all large enough $n$ (depending on $x$ and $A$), one
has $P^n(x,A)>0$ (where $P^n$ denotes the kernel of the Markov chain at time
$n$). Other definitions of irreducibility only require this property to hold
for almost every $x$ (in this case, one can restrict to an absorbing set of
full $\pi$-measure to obtain it for all $x$ there), we follow the definition
of~\cite{meyn_tweedie}.

We will be interested in a specific class of such Markov chains, called
geometrically ergodic. There are many equivalent definitions of this class,
in terms of the properties of the return time to a nice set, or of mixing
properties. Essentially, geometrically ergodic Markov chains are those Markov
chains that mix exponentially fast, see~\cite[Chapters~15
and~16]{meyn_tweedie} for several equivalent characterizations. For instance,
they can be defined as follows~\cite[Theorem 15.0.1(ii)]{meyn_tweedie}.

\begin{definition}
\label{defn:geom_erg} An irreducible aperiodic positive Markov chain is
\emph{geometrically ergodic} if the tails of the return time to some small
set are exponential. More precisely, there exist a set $C$, an integer $m>0$,
a probability measure $\nu$, and $\delta \in (0,1)$, $\kappa>1$ such that
\begin{itemize}
\item For all $x\in C$, one has
\begin{equation}
\label{eq:small}
  P^m(x,\cdot) \geq \delta \nu.
\end{equation}
\item The return time $\tau_C$ to $C$ satisfies
\begin{equation}
\label{eq:return}
\sup_{x\in C} \E_x(\kappa^{\tau_C})<\infty.
\end{equation}
\end{itemize}
\end{definition}
A set $C$ satisfying~\eqref{eq:small} is called a \emph{small} set (there is
a related notion of \emph{petite} set, these notions coincide in irreducible
aperiodic Markov chains, see~\cite[Theorem 5.5.7]{meyn_tweedie}).

In the case of a countable state space, this property is equivalent to the
fact that the return time to some  
(or equivalently any)  point has an exponential
moment.

From Theorem~15.0.1 of~\cite{meyn_tweedie}, it follows that if a chain is
geometrically ergodic in the sense of Definition~\ref{defn:geom_erg}, then
\begin{equation}\label{eq:GE}
\norm{ P^n(x, \cdot ) - \pi } \leq V(x) \rho^n \, ,
\end{equation}
where $\norm{ \cdot }$ is the total variation norm, $\rho \in (0,1)$ and $V$
is a positive function such that the set $S_V=\{x : V(x) < \infty \}$ is
absorbing and of full measure. This property~\eqref{eq:GE} is in fact another
classical definition for geometric ergodicity: from Theorem 15.4.2
in~\cite{meyn_tweedie} (or Theorem~6.14 in~\cite{nummelin2}) it follows that
if a chain is irreducible, aperiodic, positively recurrent (so that there
exists an unique stationary distribution $\pi$) and satisfies~\eqref{eq:GE},
then there exists a small set $C$ for which~\eqref{eq:return} holds.
\bigskip

We prove the following theorem.
\begin{thm}
\label{thm:weak} Let $(X_n)$ be an irreducible aperiodic Markov chain which
is geometrically ergodic on a space $\boS$.  Let $\pi$ be its stationary
distribution. Let $C$ be a small set as in Definition~\ref{defn:geom_erg}.
There exists a constant $M_0$ (depending on $C$) with the following property.
Let $n\in \N$. Let $K(x_0,\dotsc, x_{n-1})$ be a function of $n$ variables on
$\boS^n$, which is separately bounded with constants $L_i$, as
in~\eqref{eq:separ_bounded}. Then, for all $t>0$,
\begin{equation}
\label{eq:concentration}
  \Pbb_{\pi}( \abs{K(X_0,\dotsc, X_{n-1}) - \E_\pi K(X_0,\dotsc, X_{n-1})}>t) \leq 2 e^{-M_0^{-1} t^2 / \sum L_i^2} \, ,
\end{equation}
and for all $x$ in the small set $C$,
\begin{equation}
\label{eq:concentrationC}
  \Pbb_{x}( \abs{K(X_0,\dotsc, X_{n-1}) - \E_x K(X_0,\dotsc, X_{n-1})}>t) \leq 2 e^{-M_0^{-1} t^2 / \sum L_i^2} \, .
\end{equation}
\end{thm}
As will be clear from the proof, the constant $M_0$ can be written explicitly
in terms of simple numerical properties of the Markov chain, more precisely
of its coupling time and of the return time to the small set $C$. We shall in
fact prove~\eqref{eq:concentrationC}, and show how it
implies~\eqref{eq:concentration} (see the first step of the proof of
Theorem~\ref{thm:weak}).

Note that there is no strong aperiodicity assumption in our theorems (i.e.,
we are not requiring $m=1$), contrary to several mixing results for Markov
chains. The reason for this is that we will use the splitting method of
Nummelin (see Definition~\ref{def:split} below) only to control coupling
times, but we will not need the independence of the blocks between two
successive entrance times to the atom of the split chain as
in~\cite{adamczak_markov}. Following the classical strategy of McDiarmid, we
will rather decompose $K$ as a sum of martingale increments, and estimate
each of them. However, if we try to use the natural filtration given by the
time, we have no control on what happens away from $C$. The main unusual idea
in our argument is to use another filtration indexed by the next return to
$C$, the rest is mainly routine.

\bigskip

The following remarks show that the above theorem is sharp: it is not
possible to weaken the boundedness assumption~\eqref{eq:separ_bounded}, nor
the assumption of geometric ergodicity.

\begin{rmk}
It is often desirable to have estimates for functions which are unbounded. A
typical example in geometrically ergodic Markov chains is the following.
Consider an appropriate drift function, i.e., a function $V \geq 1$ which is
bounded on a small set $C$ and satisfies $PV(x) \leq \rho V(x)+A \bbun_C(x)$
for some numbers $\rho<1$ and $A\geq 0$ (where $P$ is the Markov operator of
the chain). One thinks of $V$ as being ``large close to infinity''. A natural
candidate for stronger concentration inequalities would be functions $K$
satisfying
\begin{multline}
\label{eq:V-bounded}
  \abs{K(x_0,\dotsc,x_{i-1},x_i,x_{i+1},\dotsc, x_{n-1})- K(x_0,\dotsc,x_{i-1},x'_i,x_{i+1},\dotsc, x_{n-1})}
  \\
  \leq L_i f (V(x_i) \vee V(x'_i)),
\end{multline}
for some positive function $f$ going to infinity at infinity, for instance
$f(t) = \log (1+t)$. Unfortunately, subgaussian concentration inequalities do
\emph{not} hold for such functionals  of geometrically ergodic Markov chains:
there exists a geometrically ergodic Markov chain such that, for any $M_0$,
for any function $f$ going to infinity, there exist $n$ and a functional $K$
satisfying~\eqref{eq:V-bounded} for which the
inequality~\eqref{eq:concentration} is violated. Even more, concentration
inequalities fail for additive functionals.

Consider for instance the chain on $\{1,2,\dotsc\}$ given by $\Pbb(1 \to s) =
2^{-s}$ for $s\geq 1$ and $\Pbb(s \to s-1)=1$ for $s>1$. The function $V(s) =
2^{s/2}$ satisfies the drift condition, for the small set $C=\{1\}$, since
$PV(s) = 2^{-1/2} V(s)$ for $s>1$ and $PV(1) = 2^{-1/2}/(1-2^{-1/2})<\infty$.
The stationary measure $\pi$ is given by $\pi(s)=2^{-s}$. In particular, $V$
is integrable.

Assume by contradiction that a concentration
inequality~\eqref{eq:concentration} holds for all functionals satisfying the
bound~\eqref{eq:V-bounded}, for some function $f$ going to infinity and some
$M_0>0$. Let $\tilde f$ be a nondecreasing function with $\tilde f(x) \leq
\min (f(x), x)$, tending to infinity at infinity. Define a function $g(s) =
\tilde f(V(s))$, except for $s=1$ where $g(1)$ is chosen so that $\int
g\dd\pi=0$. Let $K(x_0,\dotsc, x_{n-1})=\sum g(x_i)$, it
satisfies~\eqref{eq:V-bounded} with $L_i=L$ constant and $\E_\pi K=0$.

For any $N>0$ and $n>0$, the Markov chain has a probability $2^{-n-N}$ to
start from $X_0=n+N$, and then the next $n$ iterates are $n+N-i \geq N$. In
this case, $g(X_0)+\dotsb + g(X_{n-1}) \geq n g(N)$.
Applying~\eqref{eq:concentration}, we get
\begin{equation*}
  2^{-n-N} = \pi(n+N) \leq \Pbb_\pi( \abs{K-\E_\pi K} \geq n g(N)) \leq
  2 e^{-M_0^{-1}  (n g(N))^2 / (nL^2)} = 2 e^{-M_0^{-1} L^{-2} g(N)^2 n}.
\end{equation*}
Letting $n$ tend to infinity, we deduce that $M_0^{-1}L^{-2} g(N)^2 \leq \log
2$. This is a contradiction if $N$ is large enough, since $g$ tends to
infinity.

For instance, if one takes $f(t)=\sqrt{\ln(t\vee e)}$, then $g$ satisfies the
subgaussian condition $\E_\pi(\exp(g(X_0)^2)) < \infty$, but nevertheless the
subgaussian inequality for the additive functional $g(X_0)+\dotsb +
g(X_{n-1})$ fails.
\end{rmk}

\begin{rmk}
\label{rmk:characterize}
One may wonder if the subgaussian concentration
inequality~\eqref{eq:concentration} can be proved in larger classes of Markov
chains. This is not the case:~\eqref{eq:concentration} \emph{characterizes}
geometrically ergodic Markov chains, as we now explain.

Consider an irreducible aperiodic Markov chain such
that~\eqref{eq:concentration} holds for any separately bounded functional. We
want to prove that it is geometrically ergodic. By~\cite[Theorem
5.2.2]{meyn_tweedie}, there exists a small set, i.e., a set $C$
satisfying~\eqref{eq:small}, for some $m\geq 1$. If the original chain
satisfies subgaussian concentration inequalities, then the chain at times
which are multiples of $m$ (called its $m$-skeleton) also does. Moreover, an
irreducible aperiodic Markov chain is geometrically ergodic if and only if
its $m$-skeleton is, by~\cite[Theorem 15.3.6]{meyn_tweedie}. It follows that
is suffices to prove the characterization when $m=1$, which we assume from
now on.

The proof uses the \emph{split chain} of Nummelin (see~\cite{nummelin}
and~\cite{nummelin2}), which we describe now.
\begin{definition}\label{def:split}
Let $P$ be  a transition kernel satisfying~\eqref{eq:small} for $\delta \in
(0,1)$ and $\nu$ a probability measure.  The split chain is a Markov chain
$Y_n$ on $\bar{\boS}=\boS \times [0,1]$, whose transition kernel $\bar P$ is
as follows: if $x\notin C$, then $\bar P((x,t),\cdot) = P(x,\cdot)\otimes
\lambda$, where $\lambda$ is the uniform measure on $[0,1]$. If $x \in C$,
then if $t\in [0,\delta]$ one sets $\bar P((x,t),\cdot)=\nu\otimes \lambda$,
and if $t\in (\delta,1]$ then $\bar P((x,t),\cdot) = (1-\delta)^{-1}(\bar
P(x,\cdot)-\delta \nu)\otimes \lambda$.
\end{definition}
Essentially, the corresponding chain behaves as the chain on $\boS$, except
when it enters $C$ where the part of the transition kernel corresponding to
$\delta\nu$ is explicitly separated from the rest.

For $x\in \boS$, let $\Pbb_{\bar x}$ denote the distribution of the Markov
chain $Y_n$ started from $\delta_x\otimes \lambda$. The first component of
$Y_n$, living on $\boS$, is then distributed as the original Markov chain
started from $x$. In the same way, the chain $Y_n$ started from
$\bar\pi=\pi\otimes \lambda$ has a first projection which is distributed as
the original Markov chain started from $\pi$. For obvious reasons, we still denote by $X_n$ the first component of $Y_n$.

Let $\bar C= C\times [0,\delta]$. This is an atom of the chain $Y_n$, i.e.,
$\bar P(y, \cdot)$ does not depend on $y\in C$. We will show that the return
time $\tau_{\bar C}$ to $\bar C$ has an exponential moment. Let $C'=C\times
[0,1]$, and let $U_n$ be the second component of $Y_n$. Each time the chain
$X_n$ enters $C$, i.e., $Y_n$ enters $C'$, then $Y_n$ enters $\bar C$ if and
only if $U_n\leq \delta$. Denote by $t_k$ the $k$-th visit to $C'$ of the
chain $Y_n$, and note that $(t_k)$ is an increasing sequence of stopping
times. By the strong Markov property, it follows that $(U_{t_k})$ is an
i.i.d.\ sequence of random variables with common distribution $\lambda$. Let
$K(X_1,\dotsc,X_n)=\sum_{i=1}^n \bbun_C(X_i)$ denote the number of visits of
$X_i$ to $C$. For any $k\leq n$, $\{ K(X_1,\dotsc,X_n)\geq k \}=\{t_k \leq
n\}$.  It follows that, for any $k\leq n$,
\begin{align*}
  \Pbb_{\bar \pi}( \tau_{\bar C}>n) &\leq \Pbb_{\pi}(K(X_1,\dotsc, X_n) < k)+\Pbb_{\bar \pi}(t_k \leq n, \tau_{\bar C}>n) \\
  &\leq \Pbb_{\pi}(K(X_1,\dotsc, X_n) < k)+\Pbb_{\bar \pi}(t_k \leq n,
  U_{t_1}>\delta, \ldots , U_{t_k}>\delta)\\
&\leq \Pbb_{\pi}(K(X_1,\dotsc, X_n) < k)+(1-\delta)^k \,  .
\end{align*}
Take $k=\epsilon n$ for $\epsilon=\pi(C)/2 < \pi(C)$. The subgaussian
concentration inequality~\eqref{eq:concentration} applied to $K$ gives, for
some $c>0$, the inequality $\Pbb_{\pi}(K(X_1,\dotsc, X_n) \leq \epsilon n)
\leq 2 e^{-cn}$. We deduce that $\tau_{\bar C}$ has an exponential moment, as
desired, first for $\bar \pi$, then for its restriction to $\bar C$ since
$\bar\pi(\bar C)>0$, and then for any point in $\bar C$ since it is an atom (i.e., all starting points in $\bar C$
give rise to a chain with the same
distribution after time $1$). Hence, for some $\kappa>1$,
\begin{equation*}
  \sup_{y\in \bar C} \E_y (\kappa^{\tau_{\bar C}}) < \infty.
\end{equation*}

By definition, this shows that the extended chain $Y_n$ is geometrically
ergodic in the sense of Definition 1. It is then easy to deduce that $X_n$
also is, as follows. By~\eqref{eq:GE}, there exists a measurable function
$\bar V$ which is finite $\pi$-almost everywhere such that
\[
\norm{\bar P^n(y,\cdot)-\bar\pi} \leq \bar V(y) \rho^{n}
\]
for $\rho \in (0,1)$ and all $y$. We may take $\bar V(y) = \sup_{n\geq 1}
\rho^{-n} \norm{\bar P^n(y,\cdot)-\bar\pi}$. For $x\not \in C$, this function
$\bar V$ is constant on $\{x\}\times [0,1]$ since the chains $Y_n$ starting
from $(x,t)$ or $(x,t')$ have the same distribution after time $1$. In the
same way, for $x\in C$, the function $\bar V$ is constant on $\{x\}\times
[0,\delta]$ and on $\{x\}\times (\delta,1]$. In particular, $\bar V$ is
bounded, hence integrable, on $\pi$-almost every fiber $\{x\}\times [0,1]$.
Letting $V(x)=\int \bar V(x,t)\dd t$, we get $\norm{(\delta_x \otimes
\lambda) \bar P^n - \bar \pi} \leq V(x) \rho^n$ (we use the standard
notation: for any measure $\nu$ on $\bar{\boS}$, the measure $\nu \bar P^n$
on $\bar{\boS}$ is defined by $\nu \bar P^n(A)= \int \bar P^n(y,A) \nu(dy)$).
Since the first marginal of the chain $Y_n$ started from $\delta_x \otimes
\lambda$ is $X_n$ started from $x$, this yields $\norm{P^n(x,\cdot)-\pi} \leq
V(x)\rho^n$, where $V$ is finite $\pi$-almost everywhere. As we already
mentioned, this implies that the chain is geometrically ergodic in the sense
of Definition 1, by Theorem 15.4.2 in~\cite{meyn_tweedie}.
\end{rmk}

\bigskip

For the proof of Theorem~\ref{thm:weak}, we will use the following coupling
lemma. It says that the chains starting from any point in $C$ or from the
stationary distribution can be coupled in such a way that the coupling time
has an exponential moment.

Let us first be more precise about what we call a coupling time. In general,
a \emph{coupling} between two random variables $U$ and $V$ is a way to
realize these two random variables on a common probability space, usually to
assert some closeness property between them. Formally, it is a probability
space $\Omega^*$ together with two random variables $U^*$ and $V^*$ on
$\Omega^*$, distributed respectively like $U$ and $V$. Abusing notations, we
will usually implicitly identify $U$ and $U^*$, and $V$ and $V^*$.

Let $\mu$ and $\tilde\mu$ be two initial distributions on $\boS$. They give
rise to two chains $X_n$ and $\tilde X_n$. We will construct couplings
$(X_n^*)$ and $(\tilde X_n^*)$ between these two chains with the following
additional property: there exists a random variable $\tau:\Omega^* \to \N$,
the \emph{coupling time}, such that $X_n^* = \tilde X_n^*$ for all $n\geq
\tau$.

\begin{lem}
\label{lem:coupling} Consider an irreducible aperiodic geometrically ergodic
Markov chain and a small set $C$ as in Definition~\ref{defn:geom_erg}. There
exist constants $M_1>0$ and $\kappa>1$ with the following property. Fix $x\in
C$. Consider the Markov chains $X_n$ and $X'_n$ starting respectively from
$x$, and from the stationary measure $\pi$. Then there exists a coupling
between them with a coupling time $\tau$ such that
\begin{equation*}
\E\paren*{\kappa^\tau} \leq M_1.
\end{equation*}
\end{lem}
While this lemma has a very classical flavor, we have not been able to locate
a precise reference in the literature. We stress that the constants $\kappa$
and $M_1$ are uniform, i.e., they do not depend on $x\in C$.
\begin{proof}
We will first give the proof when the chain is strongly aperiodic, i.e., $m$
in Definition~\ref{defn:geom_erg} is equal to $1$. Then, we will deduce the
general case from the strongly aperiodic one.
%

Assume $m=1$. We use the split chain $Y_n$ on $\bar{\boS}=\boS\times [0,1]$
introduced in Definition~\ref{def:split}. We will use the notations of
Remark~\ref{rmk:characterize}, in particular $\bar C=C\times [0,\delta]$ and
$\bar\pi=\pi\otimes \lambda$ is the stationary distribution for $Y_n$. Every
time the Markov chain $X_n$ on $\boS$ returns to $C$, there is by definition
a probability $\delta$ that the lifted chain $Y_n$ enters $\bar C$. Hence, it
follows from~\eqref{eq:return} that, for some $\kappa_1>1$,
\begin{equation}
\label{eq:unif_Cbar}
  \sup_{(x,s)\in C \times [0,1]} \E_{(x,s)}(\kappa_1^{\tau_{\bar C}}) <\infty.
\end{equation}
In the same way, the entrance time to $C$ starting from $\pi$ has an
exponential moment, by Theorem 2.5 (i) in~\cite{nummelin_tuominen}.  It
follows that, for some $\kappa_2>1$,
\begin{equation}
\label{eq:unif_pi}
  \E_{\bar \pi} (\kappa_2^{\tau_{\bar C}}) < \infty.
\end{equation}

Define $T_0= \inf \{ n >0 : Y_n \in \bar C \}$ and the return times
\begin{equation*}
 T_0 +  \dotsb + T_{i+1}=
 \inf \{ n > T_0 +  \dotsb + T_{i} \st Y_n \in \bar C \} .
\end{equation*}
Then $T_0$ is independent of $(T_i)_{i>0}$ and $T_1, T_2, \dotsc$ are i.i.d.
Denote by $\Pbb_{\bar\pi}$ the probability measure on the underlying space
starting from the invariant distribution $\bar\pi$, and by $\Pbb_{\bar x}$
the probability measure starting from $\delta_x \otimes \lambda$ for $x\in
\boS$: the corresponding Markov chains lift the Markov chains on $\boS$
starting from $\pi$ and $x$ respectively. We infer from~\eqref{eq:unif_Cbar}
and~\eqref{eq:unif_pi} that there exist $\kappa_3>1$ and $M<\infty$ such that
\begin{equation}\label{eq:condlin}
\sup_{x\in C} \E_{\bar x}(\kappa_3^{T_0}) \leq M, \quad \E_{\bar\pi}( \kappa_3^{T_0}) \leq M \quad \text{and} \quad
\E( \kappa_3^{T_1})\leq M.
\end{equation}

Let now $Y_n$ and $Y'_n$ be the Markov chains on $\bar{\boS}$ where $Y_0 \sim
\delta_x \otimes \lambda$ with $x\in C$, and  $Y'_0 \sim \bar\pi$. It follows
from~\eqref{eq:condlin} that their respective return times $T_0+\dotsb+T_i$
and $T'_0+\dotsb+T'_i$ to $\bar C$ are such that:
\begin{itemize}
\item Both $T_0$ and $T'_0$ have a uniformly bounded exponential moment,
    i.e., $\E(\kappa_3^{T_0}) \leq M$ and $\E(\kappa_3^{T'_0}) \leq M$.
\item The times $T_i$ and $T'_i$ for $i\geq 1$ are all independent,
    identically distributed, and their common distribution $p$ is aperiodic
    with an exponential moment.
\end{itemize}
Define $\tau$ as
\begin{equation*}
  \tau=\inf\{n\geq 0 \st \exists i \text{ with }n=T_0+\dotsb+T_i\text{ and }\exists j\text{ with }n=T'_0+\dotsb +T'_j\}+1.
\end{equation*}
Lindvall \cite[Page 66]{lindvall} proves that, under the above two assumptions, $\tau$
has an exponential moment: there exist $\kappa<1$, $M_1<\infty$, depending
only on $\kappa_3$, $M$ and $p$, such that $\E(\kappa^{\tau}) \leq M$.

Let $Y_n^* = Y_n$ if $n<\tau$ and $Y_n^* = Y'_n$ if $n\geq \tau$. As both
$Y_{\tau-1}$ and $Y'_{\tau-1}$ both belong to the atom $\bar C$ by definition
of $\tau$, the strong Markov property shows that $(Y_n^*)_{n\in \N}$ is
distributed as $(Y_n)_{n\in \N}$. Hence, we have constructed a coupling
between $Y_n$ and $Y'_n$, with a coupling time $\tau$ which has an
exponential moment, uniformly in $x$. Considering their first marginals, this
yields the desired coupling between $X_n$ (the Markov chain on $\boS$ started
from $x$) and $X'_n$ (the Markov chain on $\boS$ started from $\pi$). This
concludes the proof when $m=1$.

Assume now $m>1$. In this case, one uses the $m$-skeleton of the original
Markov chain, i.e., the Markov chain at times in $m\N$. By~\cite[Theorem
15.3.6]{meyn_tweedie}, this $m$-skeleton is still geometrically ergodic, and
the return times to $C$ have a uniformly bounded exponential moment. Hence,
the result with $m=1$ yields a coupling between the chains $(X_{mn})_{n\in
\N}$ and $(X'_{mn})_{n\in \N}$ started respectively from $x\in C$ and from
$\pi$, with a coupling time $\tau$ having a uniformly bounded exponential
moment. Thus, we deduce a coupling between $(X_n)_{n\in \N}$ and
$(X'_n)_{n\in \N}$ together with a random variable $\tau$ taking values in
$m\N$, such that $X_{nm}= X'_{nm}$ for all $nm \in [\tau,+\infty) \cap m\N$
(from the technical point of view, this follows by seeing the fact that
$(X_{nm})$ is a subsequence of $(X_n)$ as a coupling between these two
sequences, and then using the transitivity of couplings given by Lemma A.1 of
\cite{berkes_philipp}). This is not yet the desired coupling since there is
no guarantee that $X_i = X'_i$ for $i\geq \tau$, $i\notin m\N$. Let $X_i^*
=X_i$ for $i<\tau$, and $X_i^* = X'_i$ for $i>\tau$. It is distributed as
$(X_n)$ by the strong Markov property since $X_\tau=X'_\tau$, and satisfies
$X_i^* = X'_i$ for all $i\geq \tau$ as desired.
\end{proof}

The following lemma readily follows.

\begin{lem}
\label{lem_coupling} Under the assumptions of Lemma~\ref{lem:coupling}, let
$K(x_0,\dotsc)$ be a function of finitely or infinitely many variables,
satisfying the boundedness condition~\eqref{eq:separ_bounded} for some
constants $L_i$. Then, for all $x\in C$,
\begin{equation*}
  \abs*{\E_x( K(X_0,X_1,\dotsc)) - \E_\pi(K(X_0,X_1,\dotsc))} \leq M_1 \sum_{i\geq 0}L_i \rho^i,
\end{equation*}
where $M_1>0$ and $\rho<1$ do not depend on $x$ or $K$.
\end{lem}
\begin{proof}
Consider the coupling given by the previous lemma, between the Markov chain
$X_n$ started from $x$ and the Markov chain $X'_n$ started from the
stationary distribution $\pi$. Replacing successively $X'_i$ with $X_i$ for
$i<\tau$, we get
\begin{equation*}
  \abs{K(X_0,X_1,\dotsc)-K(X'_0,X'_1,\dotsc)}
  \leq \sum_{i<\tau} L_i.
\end{equation*}
Taking the expectation, we obtain
\begin{align*}
  \abs*{\E( K(X_0,X_1,\dotsc)) - \E(K(X'_0,X'_1,\dotsc))} & \leq
    \E\paren*{\sum_{i<\tau} L_i}
  =\sum_i L_i \Pbb(\tau>i)
  \\&
  \leq \sum_i L_i \kappa^{-i} \E(\kappa^\tau)
  \leq M_1 \sum L_i \kappa^{-i}.
  \qedhere
\end{align*}
\end{proof}

\bigskip

We start the proof of Theorem~\ref{thm:weak}. To simplify the notations,
consider $K$ as a function of infinitely many variables, with $L_i=0$ for
$i\geq n$. We start with several simple reductions in the first steps, before
giving the real argument in Step 5.

\medskip

\emph{First step: It suffices to prove~\eqref{eq:concentrationC}, i.e., the
concentration estimate starting from a point $x_0\in C$.}

Indeed, fix some large $N>0$, and consider the function
\[
K_N(x_0,\dotsc,
x_{n+N-1}) = K(x_N \breakdots x_{N+n-1}).
\]
It satisfies $L_i(K_N)=0$ for $i<N$ and $L_i(K_N)=L_{i-N}(K)$ for $N\leq i<
n+N$. In particular, $\sum L_i(K_N)^2=\sum L_i(K)^2$. Applying the
inequality~\eqref{eq:concentrationC} to $K_N$, we get
\begin{equation}\label{eq:x01}
  \Pbb_{x_0}(\abs{K(X_N,\dotsc, X_{N+n-1}) - \E_{x_0} K(X_N,\dotsc, X_{N+n-1})}>t)
  \leq 2 e^{-M_0^{-1} t^2/\sum_{i\geq 0} L_i^2}.
\end{equation}
 Let
$$
g_n(x)={\mathbb E}(K(X_0,\dotsc, X_{n-1})|X_0=x)={\mathbb E}(K(X_N,\dotsc, X_{N+n-1})|X_N=x) \, .
$$
When $N\to \infty$, the distribution of $X_N$ converges towards $\pi$ in
total variation, by~\eqref{eq:GE}.
Since $g_n$ is bounded, it follows that
$$\E_{x_0} K(X_N,\dotsc, X_{N+n-1})= \E_{x_0} g_n(X_N)\to \E_\pi g_n(X_0)=
\E_\pi K(X_0,\dotsc, X_{n-1}) \quad \text{as $N \rightarrow \infty$.}
$$ 
Hence,
for
any $\epsilon>0$, their difference is bounded by $\epsilon$ if $N$ is large
enough. We obtain
\begin{align*}
  \Pbb_\pi(\abs{K(X_0,\dotsc, & X_{n-1}) -  \E_\pi K(X_0,\dotsc, X_{n-1})} > t)
  \\ &\leq \Pbb_\pi(\abs{K(X_0,\dotsc, X_{n-1}) - \E_{x_0} K(X_N,\dotsc, X_{N+n-1})} > t-\epsilon)
  \\ &\leq \epsilon + \Pbb_{x_0}(\abs{K(X_N,\dotsc, X_{N+n-1}) - \E_{x_0} K(X_N,\dotsc, X_{N+n-1})} > t-\epsilon),
\end{align*}
using again the fact that the total variation between $\pi$ and the
distribution of $X_N$ starting from $x_0$ is bounded by $\epsilon$.
Using~\eqref{eq:x01} and letting then $\epsilon$ tend to $0$, we obtain the
desired concentration estimate~\eqref{eq:concentration} starting from $\pi$,
i.e.,
\begin{equation*}
  \Pbb_\pi(\abs{K(X_0,\dotsc, X_{n-1}) - \E_\pi K(X_0,\dotsc, X_{n-1})} \geq t)
  \leq 2 e^{-M_0^{-1} t^2/\sum_{i\geq 0} L_i^2}.
\end{equation*}

\medskip

\emph{Second step: It suffices to prove that, for $x_0 \in C$,
\begin{equation}
\label{eq:ineq_exp_startC}
  \E_{x_0}(e^{K - \E_{x_0} K}) \leq e^{M_2 \sum_{i\geq 0} L_i^2},
\end{equation}
for some constant $M_2$ independent of $K$.}

Indeed, assume that this holds. Then, for any $\lambda >0$,
\begin{equation*}
  \Pbb_{x_0} ( K-\E_{x_0} K > t) \leq \E_{x_0} (e^{\lambda K -\lambda \E_{x_0} K - \lambda t})
  \leq e^{-\lambda t} e^{\lambda^2 M_2 \sum_{i\geq 0}L_i^2},
\end{equation*}
by~\eqref{eq:ineq_exp_startC}. Taking $\lambda=t/(2M_2 \sum L_i^2)$, we get a
bound $e^{-t^2/(4M_2 \sum L_i^2)}$. Applying also the same bound to $-K$, we
obtain
\begin{equation*}
  \Pbb_{x_0}( \abs{K-\E_{x_0} K}>t) \leq 2 e^{-\frac{t^2}{4M_2 \sum L_i^2}},
\end{equation*}
as desired.

\medskip

\emph{Third step: Fix some $\epsilon_0>0$. It suffices to
prove~\eqref{eq:ineq_exp_startC} assuming moreover that each $L_i$ satisfies
$L_i\leq \epsilon_0$.}

Indeed, assume that~\eqref{eq:ineq_exp_startC} is proved whenever $L_i(K)
\leq \epsilon_0$ for all $i$. Consider now a general function $K$. Take an
arbitrary point $x_* \in \boS$. Define a new function $\tilde K$ by
\begin{equation*}
  \tilde K(x_0,\dotsc, x_{n-1})=K(y_0,\dotsc, y_{n-1}),
\end{equation*}
where $y_i=x_i$ if $L_i(K)\leq \epsilon_0$, and $y_i=x_*$ if
$L_i(K)>\epsilon_0$. This new function $\tilde K$ satisfies $L_i(\tilde K) =
L_i(K) \bbun(L_i(K) \leq \epsilon_0) \leq \epsilon_0$. Therefore, it
satisfies~\eqref{eq:ineq_exp_startC}. Moreover, $\abs{K-\tilde K} \leq
\sum_{L_i(K)>\epsilon_0} L_i(K) \leq \sum L_i(K)^2/\epsilon_0$. Hence,
\begin{equation*}
  \E_{x_0}(e^{K - \E_{x_0} K}) \leq e^{2 \sum L_i(K)^2/\epsilon_0}
    \E_{x_0}(e^{\tilde K - \E_{x_0} \tilde K})
  \leq e^{2 \sum L_i(K)^2/\epsilon_0} e^{M_2 \sum L_i(\tilde K)^2}.
\end{equation*}
This is the desired inequality.

\medskip

Let us now start the proof of~\eqref{eq:ineq_exp_startC} for a function $K$
with $L_i \leq \epsilon_0$ for all $i$. We consider the Markov chain
$X_0,X_1,\dotsc$ starting from a fixed point $x_0\in C$. We define a stopping
time $\tau_i=\inf\{n\geq i \st X_n \in C\}$. Let $\boF_i$ be the
$\sigma$-field corresponding to this stopping time: an event $A$ is
$\boF_i$-measurable if, for all $n$, $A\cap \{\tau_i=n\}$ is measurable with
respect to $\sigma(X_0,\dotsc, X_n)$. Let
\begin{equation*}
  D_i=\E(K \given \boF_i) - \E(K \given \boF_{i-1}).
\end{equation*}
It is $\boF_i$-measurable. By definition of $D_i$,
\[
K(X_0,\dotsc)-\E_{x_0} (K(X_0,\dotsc))=\sum_{i=1}^n D_i  \, .
\]

\emph{Fourth step: It suffices to prove that
\begin{equation}
\label{eq:ineq_Di}
  \E(e^{D_i} \given \boF_{i-1}) \leq e^{M_3 \sum_{k\geq i} L_k^2 \rho^{k-i}},
\end{equation}
for some $M_3>0$ and some $\rho<1$, both independent of $K$.}

Indeed, assume that this inequality holds. Conditioning successively with
respect to $\boF_n$, then $\boF_{n-1}$, and so on, we get
\begin{equation*}
  \E(e^{K-\E K})=\E(e^{\sum D_i}) \leq e^{M_3 \sum_{i=0}^n \sum_{k\geq i}L_k^2 \rho^{k-i}}
  \leq e^{M_3/(1-\rho) \cdot \sum_i L_i^2}.
\end{equation*}
This is the desired inequality.

\medskip

\emph{Fifth step: Proof of~\eqref{eq:ineq_Di}.}

Note first that on the set $\{\tau_{i-1}>i-1\}$ one has $\tau_{i-1}= \tau_i$, and consequently
$D_i=0$.  Hence, the following decomposition holds:
\begin{equation}\label{eq:expressDi}\begin{split}
  D_i&= \sum_{j=i}^\infty (\E(K \given \boF_i) - \E(K \given \boF_{i-1}))\bbun_{\tau_i=j, \tau_{i-1}=i-1}
  \\
  &= \sum_{j=i}^\infty  (g_j(X_0, \ldots, X_j) -g_{i-1}(X_0, \ldots, X_{i-1}))\bbun_{\tau_i=j, \tau_{i-1}=i-1},
\end{split}\end{equation}
where
\begin{equation*}
g_j(x_0, \ldots, x_j)= \E_{X_0=x_j} K(x_0, \ldots, x_j, X_1, \ldots, X_{n-j-1}).
\end{equation*}
Here, we have used the fact that
\[
\E(K \given \boF_i)\bbun_{\tau_i=j}=\E(K
\bbun_{\tau_i=j} \given \boF_i)= \E(K  \bbun_{\tau_i=j} \given X_0, \ldots,
X_j) =\E(K  \given X_0, \ldots, X_j)\bbun_{\tau_i=j},
\]
which is commonly used in the proof of the strong Markov property for
stopping times.

Let now
\begin{equation*}
g_{j, \pi}(x_0, \ldots, x_j)= \E_{X_0\sim \pi} K(x_0, \ldots, x_j, X_1, \ldots, X_{n-j-1}) .
\end{equation*}
By Lemma~\ref{lem_coupling}, for any $x_j \in C$,
\begin{equation}\label{eq:ineqg1}
\abs{g_j(x_0, \ldots, x_j)-g_{j, \pi}(x_0, \ldots, x_j)} \leq M_1 \sum_{k\geq j+1} L_k \rho^{k-j}  .
\end{equation}
From~\eqref{eq:expressDi} and~\eqref{eq:ineqg1}, we infer that
\begin{equation}\label{eq:ineqDi2}\begin{split}
  D_i={}&\sum_{j=i}^\infty  (g_{j, \pi}(X_0, \ldots, X_j) -g_{i-1, \pi}(X_0, \ldots, X_{i-1}))\bbun_{\tau_i=j, \tau_{i-1}=i-1}
  \\&+O\paren*{ \sum_{k\geq \tau_i+1} L_k \rho^{k-\tau_i}} +O\paren*{\sum_{k\geq i} L_k \rho^{k-i}}.
\end{split}\end{equation}
Since $\pi$ is the stationary measure, $g_{j, \pi}$ can also be written as
\begin{equation*}
  g_{j, \pi}(x_0, \ldots, x_j)= \E_{X_0\sim \pi} K(x_0, \ldots, x_j, X_{j-i+2}, \ldots, X_{n-i}) .
\end{equation*}
It follows that
\begin{equation}\label{eq:ineqg2}
  \abs{g_{j, \pi}(x_0, \ldots, x_j) -g_{i-1, \pi}(x_0, \ldots, x_{i-1})} \leq \sum_{k=i}^{j} L_k     .
\end{equation}
Write $\tau=\tau_i-(i-1)$ for the return time to $C$ of $X_{i-1}$.
From~\eqref{eq:ineqg2}, we get that
\begin{equation}\label{eq:ineqg3}
\sum_{j=i}^\infty  (g_{j, \pi}(X_0, \ldots, X_j) -g_{i-1, \pi}(X_0, \ldots, X_{i-1}))\bbun_{\tau_i=j, \tau_{i-1}=i-1}
\leq \left (\sum_{k=i}^{i+\tau-1} L_k \right ) \bbun_{ \tau_{i-1}=i-1} .
\end{equation}
Since $\sum_{k\geq i} L_k \rho^{k-i} \leq \sum_{k=i}^{i+\tau-1} L_k + \sum_{k \geq i+\tau} L_k \rho^{k-i-\tau}$, it follows
from~\eqref{eq:ineqDi2} and~\eqref{eq:ineqg3} that
  \begin{equation}
  \label{eq:main_bound_Di}
  \abs{D_i} \leq M_4 \left (\sum_{k=i}^{i+\tau-1} L_k + \sum_{k\geq i+\tau} L_k \rho^{k-i-\tau}\right ) \bbun_{ \tau_{i-1}=i-1} .
  \end{equation}
As all the $L_k$ are bounded by $\epsilon_0$, we obtain
\begin{equation}
\label{eq:first_bound_Di}
  \abs{D_i} \leq M_4 \epsilon_0(\tau+1/(1-\rho))\bbun_{ \tau_{i-1}=i-1} \leq M_5\epsilon_0 \tau \bbun_{ \tau_{i-1}=i-1}.
\end{equation}
Choose $\sigma \in [\rho, 1)$. The equation~\eqref{eq:main_bound_Di} also
gives
\begin{equation*}
  \abs{D_i} \leq M_4 \left (\sum_{k\geq i} L_k \sigma^{k-i} \sigma^{-\tau} \right ) \bbun_{ \tau_{i-1}=i-1}.
  \end{equation*}
By the Cauchy-Schwarz inequality, this yields
\begin{equation}
\label{eq:second_bound_Di}
\begin{split}
  \abs{D_i}^2 & \leq M_4^2 \sigma^{-2\tau} \paren*{\sum_{k\geq i} L_k^2 \sigma^{k-i}} \paren*{ \sum_{k\geq i} \sigma^{k-i}} \bbun_{ \tau_{i-1}=i-1}
  \\ & \leq M_6
  \sigma^{-2\tau} \left (\sum_{k\geq i} L_k^2 \sigma^{k-i}\right ) \bbun_{ \tau_{i-1}=i-1}.
\end{split}
\end{equation}

We have $e^t \leq 1+t+ t^2 e^{\abs{t}}$ for all real $t$. Applying this
inequality to $D_i$, taking the conditional expectation with respect to
$\boF_{i-1}$ and using that $\E(D_i \given \boF_{i-1})=0$, this gives
  \begin{equation*}
  \E(e^{D_i} \given \boF_{i-1}) \leq 1+\E(D_i^2 e^{\abs{D_i}} \given \boF_{i-1}).
  \end{equation*}
Combining this estimate with~\eqref{eq:first_bound_Di}
and~\eqref{eq:second_bound_Di}, we get
\begin{align*}
  \E(e^{D_i} \given \boF_{i-1}) &
  \leq 1+\E\paren*{ M_6 e^{M_5 \epsilon_0 \tau}\sigma^{-2\tau} \sum_{k\geq i}L_k^2 \sigma^{k-i} \given \boF_{i-1}}
  \bbun_{ \tau_{i-1}=i-1} \\
  &
  \leq 1+ M_6 \left (\sum_{k\geq i}L_k^2 \sigma^{k-i}\right)\E\paren*{ e^{M_5 \epsilon_0 \tau}\sigma^{-2\tau}  \given X_{i-1}}
  \bbun_{ X_{i-1}\in C}.
\end{align*}
By the definition of geometric ergodicity (see Definition~\ref{defn:geom_erg}) one can
choose $\epsilon_0$ small enough and $\sigma$ close enough to $1$ in such a way that
\[
\sup_{x \in C} \E\paren*{ e^{M_5 \epsilon_0 \tau}\sigma^{-2\tau}  \given X_{i-1}=x} < \infty \, .
\]
It follows that
\begin{equation*}
  \E(e^{D_i} \given \boF_{i-1}) \leq 1+M_7 \sum_{k\geq i}L_k^2 \sigma^{k-i}
  \leq e^{M_7 \sum_{k\geq i}L_k^2 \sigma^{k-i}}.
\end{equation*}
This concludes the proof of~\eqref{eq:ineq_Di}, and of
Theorem~\ref{thm:weak}. \qed

\bibliography{biblio}
\bibliographystyle{amsalpha}
\end{document}